\documentclass[12pt, a4paper, reqno]{amsart}
\usepackage{amssymb}
\usepackage{amsbsy}
\usepackage{xcolor}
\usepackage[normalem]{ulem}
\usepackage[hdivide={2.5cm,,2.5cm}, vdivide={3.5cm,,2.8cm}]{geometry}
\usepackage{amssymb}
\usepackage{amsthm}
\usepackage{amsmath}
\usepackage{mathtools}
\usepackage[pctexwin]{graphicx}
\usepackage{stackengine}

\usepackage{pdfsync}

\makeatletter
\newcommand*{\Log}{\mathop{\operator@font Log}\nolimits}
\newcommand*{\Arg}{\mathop{\operator@font Arg}\nolimits}
\newcommand*{\tg}{\mathop{\operator@font tg}\nolimits}
\newcommand*{\ctg}{\mathop{\operator@font ctg}\nolimits}
\newcommand*{\cosec}{\mathop{\operator@font cosec}\nolimits}
\newcommand*{\arctg}{\mathop{\operator@font arctg}\nolimits}
\newcommand*{\arcctg}{\mathop{\operator@font arcctg}\nolimits}
\newcommand*{\sh}{\mathop{\operator@font sh}\nolimits}
\newcommand*{\ch}{\mathop{\operator@font ch}\nolimits}
\makeatother

\newcommand{\REM}[1]{\relax}

\numberwithin{equation}{section}

\newcommand{\UH}{\mathbb{H}}

\newcommand{\Real}{\mathbb{R}}
\newcommand{\Natural}{\mathbb{N}}
\newcommand{\Complex}{\mathbb{C}}

\newcommand{\ComplexE}{\overline{\mathbb{C}}}

\newcommand{\UD}{\mathbb{D}}
\newcommand{\clS}{\mathcal{S}}

\newcommand{\clSt}[1]{\mathrlap{\hskip.125em\widetilde{\hphantom{\clS}\vphantom{\vbox to1.5ex{\vss}}}}\clS_{#1}^B\vphantom{\vbox to2.3ex{\vss}}}

\newcommand{\clsubS}{\mathrlap{\hskip.125em\widetilde{\hphantom{\clS}\vphantom{\vbox to1.5ex{\vss}}}}\clS\vphantom{\vbox to2.3ex{\vss}}}

\newcommand{\clH}{\mathcal{H}}
\newcommand{\Cara}{\mathcal{C}}

\newcommand{\onto}
{\xrightarrow{\scriptstyle \!\mathsf{onto}\,}}

\newcommand{\into}
{\xrightarrow{\hbox{\lower.2ex\hbox{$\scriptstyle \smash{\mathsf{into}}$}}\,}}

\let\D=\UD

\let\C=\Complex

\newcommand{\STOP}{\par\hbox to\textwidth{\color{red}\leaders\hbox{\,STOP\,}\hfil}\par}

\newcommand{\mcite}[1]{\csname b@#1\endcsname}

\theoremstyle{theorem}
\newtheorem{result}{Theorem}
\def\Re{\mathop{\mathtt{Re}}}
\def\Im{\mathop{\mathtt{Im}}}

\newcommand{\sgn}{\mathop{\mathrm{sgn}}}

\newtheorem{theorem}{Theorem}[section]

\newtheorem{proposition}[theorem]{Proposition}
\newtheorem{corollary}[theorem]{Corollary}
\newtheorem{problem}{Problem}

\theoremstyle{definition}

\newtheorem{example}[theorem]{Example}

\theoremstyle{remark}
\newtheorem{remark}[theorem]{Remark}
\numberwithin{equation}{section}

\newcommand{\di}{\mathrm{d}}

\def\mydot#1{\smash{\stackrel{\,\lower.12ex\hbox{\text{\LARGE.}}}{#1}}\vphantom{\raise.2ex\hbox{$#1$}}}

\newcommand{\SHOWCORRECTIONS}{%
\newcommand{\nv}[1]{{\color{green!60!black}##1}}%
\newcommand{\comm}[1]{{\color[rgb]{0.5,0,0.5}##1}}%
\newcommand{\dv}[1]{{\color[rgb]{0.65,0.74,0.79}\sout{##1}}}%
\newcommand{\IN}[1]{{\color[rgb]{1.00,0.33,0.33}##1}}
\newcommand{\ID}[1]{{\color[rgb]{0.65,0.55,0.62}\sout{##1}}}%
\newcommand{\IC}[1]{{\color[rgb]{0.00,0.00,1}##1}}%
\newcommand{\RD}[1]{\textcolor{red}{##1}}%
}%

\newcommand{\HIDECORRECTIONS}{%
\newcommand{\nv}[1]{##1}
\newcommand{\dv}[1]{\relax}%
\newcommand{\comm}[1]{\relax}%
\let\IN=\nv%
\let\ID=\dv%
\let\IC=\comm%
\newcommand{\RD}[1]{##1}%
}%

\SHOWCORRECTIONS

        \let\bibciteOLD=\bibcite
        \renewcommand{\bibcite}[2]{\bibciteOLD{#1}{\textsc{#2}}}
        \makeatletter
        \renewcommand{\@biblabel}[1]{[\textsc{#1}]\hfill}
        \makeatother

\title[Univalent functions with quasiconformal extensions]{Univalent functions with quasiconformal extensions: Becker's class and estimates of the third coefficient.}

\author[P. Gumenyuk]{Pavel Gumenyuk}

\address{Institutt for matematikk og fysikk, Universitetet i Stavanger, 4036 Stavanger, \hbox{Norway}}
\email{pavel.gumenyuk@uis.no}
\author[I. Hotta]{Ikkei Hotta$^\dag$}
\address{Department of Applied Science, Yamaguchi University, 2-16-1 Tokiwadai, Ube 755-8611, \hbox{Japan}}
\email{ihotta@yamaguchi-u.ac.jp}
\thanks{$^\dag$ Ikkei Hotta is supported by JSPS KAKENHI Grant Number 17K14205.}
\date{\today}

\let\le=\leqslant
\let\ge=\geqslant

\newcommand{\esssup}{\mathop{\mathrm{ess\,sup}}}

\begin{document}

\begin{abstract}
We investigate univalent functions $f(z)=z+a_2z^2+a_3z^3+\ldots$ in the unit disk~$\UD$ extendible to $k$-q.c.(=quasiconformal) automorphisms of~$\C$. In particular, we answer a question on estimation of~$|a_3|$ raised by K\"uhnau and Niske [Math. Nachr. {\bf 78} (1977) 185--192]. This is one of the results we obtain studying univalent functions that admit q.c.-extensions via a construction, based on Loewner's parametric representation method, due to Becker [J. Reine Angew. Math. {\bf 255} (1972) 23--43].
Another problem we consider is to find the maximal $k_*\in(0,1]$ such that every univalent function~$f$ in~$\UD$ having a $k$-q.c. extension to~$\C$ with $k\le k_*$ admits also a Becker q.c.-extension, possibly with a larger upper bound for the dilatation. We prove that $k_*>1/6$. Moreover, we show that in some cases, Becker's extension turns out to be the optimal one.
Namely, given any $k\in(0,1)$, to each finite Blaschke product there corresponds a univalent function $f$ in~$\UD$ that admits a Becker $k$-q.c. extension but no $k'$-q.c. extensions to~$\C$ with~$k'<k$.
\end{abstract}

\maketitle

\section{Introduction}
Conformal mappings of~$\UD:=\{z:|z|<1\}$ admitting quasiconformal extensions is a classical topic in Geometric Function Theory closely related to Teichm\"uller Theory, see e.g.~\cite{Lehto_book,Schober}. Let $k\in(0,1)$. A function~$f$ holomorphic in a domain~$D\subset\Complex$ is said to be \hbox{$k$-q.c.} extendible to~$\Complex$ (or to~$\ComplexE$) if there exists a $k$-quasiconformal automorphism ${F:\Complex\to\Complex}$ (respectively, ${F:\ComplexE\to\ComplexE}$) such that $F|_\UD=f$. Note that $k$-q.c. extendibility to~$\Complex$, which we will be mostly concerned with in this paper, is equivalent to $k$-q.c. extendibility to~$\ComplexE$ with the additional condition that~$F(\infty)=\infty$.

Denote by $\clS$ the class of all univalent (i.e. injective holomorphic) functions $$\UD\ni z\mapsto f(z)=z+\sum_{n=2}^{+\infty}a_nz^n.$$ One of the main tools to study this class is the \textsl{parametric representation}, which goes back to Loewner~\cite{Loewner}, see e.g. \cite[\S6.1]{Pommerenke}, see also~\cite{Kufarev43,Pommerenke65,Gut}.
Namely, the class $\clS$ can be represented as an image of the convex cone formed by the so-called \textsl{Herglotz functions}, i.e. functions $p:\UD\times[0,+\infty)\to\Complex$ such that $p(z,\cdot)$ is locally integrable for each~$z\in\UD$ and $p(\cdot,t)$ is holomorphic in~$\UD$ and satisfies $\Re p(\cdot,t)\ge0$ for a.e.~$t\ge0$. It is known that for any Herglotz function $p$,
the initial value problem for the Loewner\,--\,Kufarev ODE
\begin{equation}\label{EQ_LK-ODE0}
\frac{\di w}{\di t}=-w\,p(w,t),\quad t\ge0,\qquad w(z,0)=z\in\UD,
\end{equation}
has a unique solution $w=w(z,t)$ and the locally uniform limit
\begin{equation}\label{EQ_f-limit}
f(z):=\lim_{t\to+\infty}\frac{\,\mathrlap{w(z,t)}\hphantom{w'(0,t)}}{\,w'(0,t)},\quad z\in\UD,
\end{equation}
where $w'$ denotes the derivative w.r.t.~$z$, exists and belongs to~$\clS$. On the other hand, see e.g.~\cite[Theorem~6.1 on p.\,159]{Pommerenke} or~\cite{Gut}, \textit{every} function~$f\in\clS$ can be represented by~\eqref{EQ_f-limit} with a suitable, and in general not unique, \textsl{normalized} Herglotz function, i.e. a Herglotz function~$p$ with~$\Re p(0,t)=1$ for a.e.~$t\ge0$.

A natural problem arises: \textit{given a subclass~$\clsubS\subset\clS$, find a class of Herglotz functions that generates~$\clsubS$ via~\eqref{EQ_f-limit}.} The answer is known in some cases, e.g. for starlike functions, bounded univalent functions, and for univalent functions with real Taylor coefficients; \hbox{see e.g.\,\cite{Prokhorov-handbook}.}

A partial answer is also known for the subclass~$\clS_k$, $k\in(0,1)$, formed by all $f\in\clS$ admitting $k$-q.c. extension to~$\Complex$. Namely, in 1972, Becker~\cite{Becker72} found a condition on $p$ in the Loewner\,--\,Kufarev equation~\eqref{EQ_LK-ODE0}, see Sect.\,\ref{SS_Becker}, such that the function $f$ given by~\eqref{EQ_f-limit} belongs to~$\clS_k$.
The class~$\clS_k^B$ generated by Herglotz functions that satisfy Becker's condition is a \textit{proper} subset of~$\clS_k$.
In this paper we study~$\clS_k^B$ and its relation with~$\clS_k$. In particular, in Sect.\,\ref{SS_a3} we find the sharp estimate for~$|a_3|$ in~$S_k^B$, see Theorem~\ref{TH_a_3}. An immediate corollary is the answer to a question of K\"uhnau and Niske~\cite{KuhnauNiske}: Theorem~\ref{TH_a_3} implies that $\max_{S_k}|a_3|>k$ for any~$k\in(0,1)$.

Numerous sharp estimates are known for the class~$\clS$, see e.g.~\cite{Duren}, with many of them being motivated by the famous Bieberbach Conjecture concerning estimates for~$|a_n|$, which was proved by de Branges~\cite{deBranges} in~1984. Unfortunately, only a few of these results have been extended to classes~$\clS_k$, see e.g. \cite{Kuhnau69,Lehto71}. In particular, the sharp estimate for $|a_n|$ in~$\clS_k$ is known only for~$n=2$. Remarkably, in most of the cases discussed previously, the extremal functions belong to~$\clS_k^B$. We prove a bit surprising fact that this does not hold for the sharp estimate of~$|a_3|$, see Theorem~\ref{TH_not-extremal}.

\section{Becker's construction of quasiconformal extensions}\label{SS_Becker}
Throughout the paper we make use of Loewner Theory, the classical version of which can be found in~\cite[Chapter~6]{Pommerenke}.
Following Becker~\cite{Becker76}, \cite[\S5.1]{Becker80}, we replace the usual normalization ${p(0,t)=1}$ by a weaker
condition
\begin{equation}\label{EQ_weaker-normalization-for-p}
\int_0^{+\infty}\!\!\Re p(0,t)\,\di t=+\infty,
\end{equation}
which still implies that $\bigcup\limits_{t\ge0}f_t(\UD)=\C$.
In 1972, he discovered the
\hbox{following remarkable fact.}
\begin{result}[\protect{\cite{Becker72,Becker76}}]\label{TH_Becker}
Let $k\in[0,1)$ and let $(f_t)$ be a radial Loewner chain whose Herglotz function~$p$ satisfies
\begin{equation}\label{EQ_Becker-condition}
  p(\UD,t) \subset U(k):=\left\{ w \in \C \colon \left|\frac{w-1}{w+1}\right| \le k\right\}\quad\text{for  a.e.~$t\ge0$}.
\end{equation}
Then for every~$t\ge0$, the function~$f_t$ admits a $k$-q.c. extension to~$\ComplexE$ that fixes~$\infty$. In particular, such an
extension for~$f_0$ is given by
\begin{equation}\label{EQ_BeckerExt}
F(\rho e^{i\theta}):=\left\{
 \begin{array}{ll}
  f_0(\rho e^{i\theta}), & \text{if~$0\le \rho<1$},\\[1ex]
  f_{\log \rho}(e^{i\theta}), & \text{if~$\rho\ge1$}.
 \end{array}
\right.
\end{equation}
\end{result}
\begin{remark}\label{RM_2018}
According to \cite[Theorem~2]{Istvan}, a sort of converse statement holds. Namely, if $(f_t)$ is a Loewner chain such that all $f_t$'s extend continuously to~$\partial\UD$ and the map $F$ defined by~\eqref{EQ_BeckerExt} is $k$-quasiconformal in~$\C$, then the Herglotz function~$p$ associated with~$(f_t)$ satisfies Becker's condition~\eqref{EQ_Becker-condition}.
\end{remark}

In what follows, for $k\in(0,1)$, we will denote by $\clS_k^B$ the class of all $f\in\clS$ admitting Loewner's representation with the Herglotz function $p$ normalized by $p(0,t)=1$ a.e.\,$t\ge0$ and satisfying~\eqref{EQ_Becker-condition}. A bit larger class of all $f\in\clS$ generated by Herglotz functions subject to Becker's condition~\eqref{EQ_Becker-condition}, but not necessarily normalized, will be denoted by~$\clSt{k}$.

According to Theorem~\ref{TH_Becker}, $\clS^B_k\subset\clSt{k}\subset\clS_k$. It is known that $\clSt{k}\neq\clS_k$, see \hbox{e.g.~\cite[\S5]{Istvan}}. However, it seems that the study of~$\clS_k^B$ and~$\clS_k$ is still of considerable interest. It is worth to mention that Becker's condition~\eqref{EQ_Becker-condition} appears to be sufficient for q.c.-extendibility also in the framework of the general Loewner Theory introduced in~\cite{BracciCD:evolutionI, BracciCD:evolutionII}; see \cite{Istvan}, \cite{Ikkei2}, and~\cite{HottaGum::QC-chordal}. This discussion will be continued in Sect.\,\ref{SS_relation}.

\section{Estimate of the third coefficient}\label{SS_a3}
Below we give a sharp estimate for $|a_3|$ in the class~$\clS_k^B$. As a corollary, we immediately obtain a \emph{negative} answer to the question raised in 1977 by K\"uhnau and Niske~\cite{KuhnauNiske}: does there exist $k_0>0$ such that for any $k\in(0,k_0]$ and any function $f(z)=z+a_2z+a_3 z^3+\ldots$ belonging to~$\clS_k$, the inequality $|a_3|\le k$ holds?

\begin{theorem}\label{TH_a_3}
Let $k\in(0,1)$. Then for every function $f(z)=z + a_2z^2 + a_3z^3+\ldots$ belonging to $\clS_k^B$,
\begin{equation*}
|a_3|\le k\big(1+e^{1-1/k}(1+k)\big).
\end{equation*}
This estimate is sharp and the equality holds only for rotations of the function $f_+\in\clS_k^B$, which is uniquely defined by the Beltrami coefficient~\eqref{EQ_extremal-Beltrami-coefficient} of its q.c.-extension to~$\C$.
\end{theorem}
The above theorem does not solve the extremal problem $|a_3|\to\max$ \emph{in the whole class~$\clS_k$}. In fact, the following takes place.
\begin{theorem}\label{TH_not-extremal}
For any $k\in(0,1)$,
\begin{equation}\label{EQ_upper-est}
\max_{\clS_k^B}|a_3|<\max_{\clS_k}|a_3|\le \varrho(k):=\min_{\alpha\in(0,1)}\Bigl[\big(1+2e^{-2\alpha/(1-\alpha)}\big)k+4\alpha k^2\Bigr].
\end{equation}
\end{theorem}

\begin{remark}\label{RM_Krushkal}
The sharp estimate in Theorem~\ref{TH_a_3} shows that the inequality ${|a_n|\le 2k/(n-1)}$ written in the larger class $\clS_k$ for $0<k\le 1/(1+n^2)$ and all $n=2,3,\ldots$ by Krushkal~\cite[Corollary on p.\,350]{Krushkal-a_n}, in fact, fails for~$n=3$. Note that the two estimates have tangency of infinite order at~$k=0$, while the difference from the r.h.s. of~\eqref{EQ_upper-est} behaves asymptotically as~$4k^2$. The three estimates are shown in Figure~\ref{FI_graph}.
\end{remark}
\begin{figure}[t]
\includegraphics[bb=0 0 543 333, width=9.5cm]{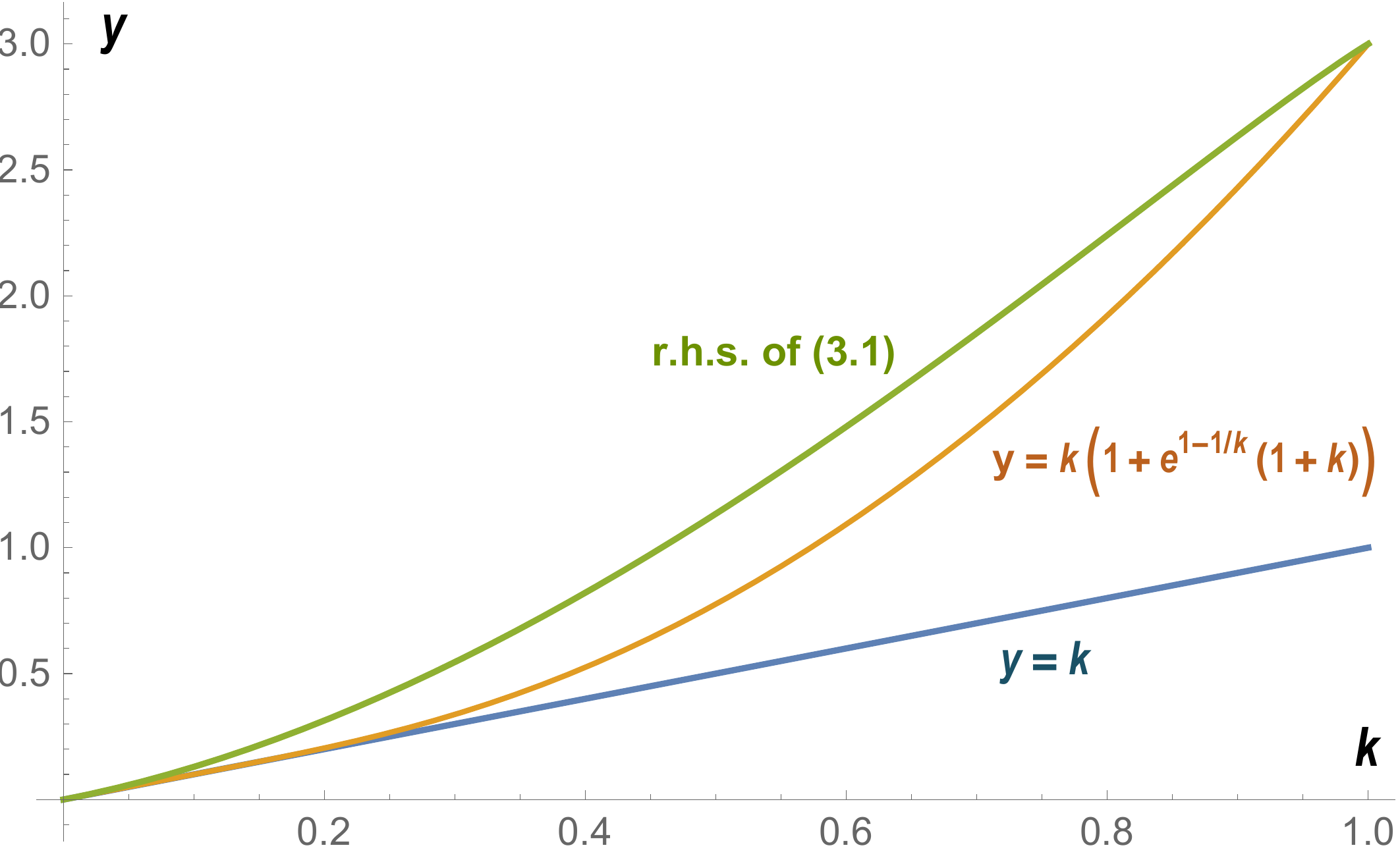}
\caption{Estimates for~$|a_3|$ mentioned in Remark~\ref{RM_Krushkal}.}\label{FI_graph}
\end{figure}
\vskip2ex
\begin{proof}[\fontseries{bx}\selectfont\protect{Proof of Theorem~\ref{TH_a_3}}]
The class $\clS_k^B$, ${k\in(0,1)}$, admits a Loewner-type parametric representation. Denote by $\clH_k$ the class of all normalized Herglotz functions $p$ satisfying $p(\UD,t)\subset U(k)$ for a.e.~$t\ge0$, where $U(k)$ is the closed disk defined in Theorem~\ref{TH_Becker}. As it follows from the very definition,  $\clS_k^B$~coincides with the image of the map
$$
\clH_k\ni p~\mapsto~f:=\lim_{t\to+\infty}e^{t}w(z,t)~\in\,\clS,
$$
where for each $z\in\UD$ the function $[0,+\infty)\ni t\mapsto w(z,t)\in\UD$ is defined as the unique solution to the initial value problem~\eqref{EQ_LK-ODE0}.
Write $p(z,t)=1+p_1(t)z+p_2(t)z^2+\ldots$ for all $z\in\UD$ and a.e.~$t\ge0$ and let
$$
f(z,t):=e^{t}w(z,t)=z+a_2(t)z^2+a_3(t)z^3+\ldots
$$
There is one-to-one correspondence between the class of all normalized Herglotz functions and~$\clH_k$. Indeed, $p\in\clH_k$ if and only if it can be written as $p(\cdot,t)=L\circ p_0(\cdot,t)$ for a.e.~$t\ge0$, where $p_0(z,t)=1+c_1(t)z+c_2(t)z^2+\ldots$ is an arbitrary normalized Herglotz function and $$L(z):=\frac{1+Kz}{K+z},\quad K:=\frac{1+k}{1-k},$$ is a conformal map of $\UH:=\{z\colon\Re z>0\}$ onto $U(k)$ with $L(1)=1$.

As usual, from~\eqref{EQ_LK-ODE0} we obtain the initial value problem for the coefficients $a_2$ and~$a_3$,
\begin{align}
\frac{\di a_2}{\di t}&=-e^{-t} p_1(t)=-ke^{-t} c_1(t), & a_2(0)=0,\label{EQ_phase1}\\[.7ex]
\frac{\di a_3}{\di t}&=-e^{-2t}p_2(t)-2e^{-t}p_1(t)a_2(t)\notag\\&=-k\Big(e^{-2t}\big(c_2(t)-(1-k)\frac{c_1(t)^2}{2}\,\big)+2e^{-t}c_1(t)a_2(t)\Big), & a_3(0)=0.\label{EQ_phase2}
\end{align}

Since along with any $f\in\clS_k^B$ the class $\clS_k^B$ contains all rotations of~$f$, i.e. the functions ${z\mapsto e^{i\theta}f(e^{-i\theta}z)}$, ${\theta\in\Real}$, the problem to determine $\max|a_3|$ in~$\clS_k^B$ is equivalent to
finding $~\max\,\Re a_3$. The latter problem can be reformulated as the optimal control problem for the above system and the objective functional $\Re a_3(+\infty)$, with a control function ${t\mapsto \big(c_1(t),c_2(t)\big)\in\C^2}$ regarded as admissible if it is measurable and for a.e.~$t\ge0$ satisfies
\begin{equation}\label{EQ_Cara}
 |c_1|\le 2,\qquad |2 c_2-c_1^2|\le 4-|c_1|^2.
\end{equation}
Conditions~\eqref{EQ_Cara} describe the value region of $\Cara\ni q\mapsto (c_1,c_2)\in\C^2$ over the Carath\'eodory class~$\Cara$ of all holomorphic functions $q(z)=1+c_1 z+c_2 z^2+\ldots$ in~$\UD$ with positive real part; see, e.g., \cite[Chapter IV, \S7]{Tsuji}.

To apply Pontryagin's Maximum Principle, we define the (holomorphic) Hamiltonian
$$
 H(a_2,a_3,\psi_2,\psi_3,t,c_1,c_2):=-ke^{-t} c_1\psi_2-k\Big(e^{-2t}\big(c_2-(1-k)\frac{c_1^2}{2}\big)+2e^{-t}c_1a_2\Big)\psi_3
$$
and write the adjoint system
\begin{align}
\frac{\di \psi_2}{\di t}&=-\frac{\partial H}{\partial a_2}=2ke^{-t}c_1(t)\psi_3(t),\label{EQ_adj1}\\[1.25ex]
\frac{\di \psi_3}{\di t}&=-\frac{\partial H}{\partial a_3}=0.\label{EQ_adj2}
\end{align}
The maximum of $\Re a_3(+\infty)$ is to be found among all the trajectories of \eqref{EQ_phase1},\,\eqref{EQ_phase2} satisfying the initial condition at~$t=0$, while the right-hand endpoint of the trajectories is variable. Therefore, according to Pontryagin's Maximum Principle, see \cite[Chapter~I, \S7, Theorem $3^*$]{Pontryagin}, if $c_1(t)=c_1^*(t)$, $c_2(t)=c_2^*(t)$ is an optimal control in our problem, then for the corresponding solution to the phase system~\eqref{EQ_phase1},\,\eqref{EQ_phase2} supplemented with the adjoint equations~\eqref{EQ_adj1},\,\eqref{EQ_adj2} and  the transversality conditions
\begin{equation}\label{EQ_trans}
\psi_2(+\infty)=0,\qquad \psi_3(+\infty)=1,
\end{equation}
it holds that
\begin{multline}\label{EQ_Max-principle}
\max_{(c_1,c_2)} \Re H\big(a_2(t),a_3(t),\psi_2(t),\psi_3(t),t,c_1,c_2\big)\\ = \Re H\big(a_2(t),a_3(t),\psi_2(t),\psi_3(t),t,c^*_1(t),c^*_2(t)\big),
\end{multline}
where the maximum is taken over all~$(c_1,c_2)\in\C^2$ subject to conditions~\eqref{EQ_Cara}.

System~\eqref{EQ_adj1}\,--\,\eqref{EQ_trans} can be integrated using integrals to~\eqref{EQ_phase1},\,\eqref{EQ_phase2}:
\begin{equation}\label{EQ_adjoint}
\psi_2(t)=a-2a_2(t),\quad \psi_3(t)=1,
\end{equation}
where $a:=2a_2(+\infty)$.

To find the maximum of $\Re H$ as a function of~$c_1$ and~$c_2$, we first fix a~$c_1\in\C$ with~$|c_1|\le2$ and optimize $\Re H$ in the disk described by the second of the inequalities in~\eqref{EQ_Cara}. The maximum is achieved for ${c_2=c^*_2:=(\Re c_1)^2+i\,\Re c_1\,\Im c_1-2}$. For this value of~$c_2$ and taking into account~\eqref{EQ_adjoint}, we get
\begin{align}
-\frac{e^{2t}}{k}\,\Re H&=e^t\,\Re(ac_1)+\frac{1+k}{2}\,c_1'^{\,2}+\frac{1-k}{2}\,c_2''^{\,2}-2\notag\\%
                        &=\,\frac{1+k}{2}\left(c_1'+\frac{e^ta'}{1+k}\right)^{\!\!2}\,%
                            +\,\frac{1-k}{2}\left(c_1''-\frac{e^ta''}{1-k}\right)^{\!\!2}\,+\,C,\label{EQ_quadratic}
\end{align}
where $a=:a'+ia''$, $c_1=:c_1'+ic_1''$, and $C$ is a quantity independent of~$c_1$. The absolute minimum of~\eqref{EQ_quadratic} is achieved at $c_1^\star:=e^t\big(-a'/(1+k)+ia''/(1-k)\big)$. Moreover, even if $|c_1^\star|>2$, the minimum point~$c_1^*$ of~\eqref{EQ_quadratic} over the disk $|c_1|\le 2$ still satisfies
\begin{equation}\label{EQ_sgn}
 \sgn\Re c_1^*=-\sgn a',\qquad \sgn\Im c_1^*=\sgn a'',
\end{equation}
where $\sgn x:=x/|x|$ for $x\in\Real\setminus\{0\}$ and $\sgn 0:=0$.
For the optimal trajectory, according to~\eqref{EQ_phase1}, we have
\begin{equation}\label{EQ_a}
 a=-2k\int_0^{+\infty}\!\!e^{-t}c_1^*(t)\,\di t,
\end{equation}
which would contradict~\eqref{EQ_sgn} whenever $a''\neq0$. Therefore, $a$ is real and
\begin{equation}\label{EQ_c1}
c_1^*(t)=
\begin{cases}
-e^ta/(1+k),& \text{if~}\big|e^ta/(1+k)\big|<2,\\
-2\sgn a,&      \text{otherwise}.
\end{cases}
\end{equation}

Consider two cases. First suppose that $a=0$. Then $c_1^*(t)=0$ and $c_2^*(t)=-2$ for all~$t\ge0$. Note that for~$(c_1,c_2)=(c_1^*,c_2^*)$ in~\eqref{EQ_Cara}, the first condition is satisfied with the strict inequality sign, while in the second condition equality occurs.
Therefore, see, e.g., \cite[Theorem~IV\!.\,23]{Tsuji},
\begin{align*}
 p_0(z,t)&=\lambda\frac{1+\mu_1 z}{1-\mu_1 z}+(1-\lambda)\frac{1+\mu_2 z}{1-\mu_2 z}\\
          &=1\,+\,2(\lambda\mu_1+(1-\lambda)\mu_2)z\,+\,2(\lambda\mu_1^2+(1-\lambda)\mu_2^2)z^2\,+\,\ldots,\quad z\in\UD,
\end{align*}
for some constants $\lambda\in(0,1)$ and $\mu_1\neq\mu_2$ on the unit circle (possibly depending on~$t$). Comparing the coefficients of $z$ and $z^2$ with $c_1^*$ and~$c_2^*$, we conclude that ${\lambda=1/2}$, ${\mu_{1,2}=\pm i}$, and hence $p(z,t)=(1-kz^2)/(1+kz^2)$. The corresponding function $f\in\clS_k^B$ is $f(z)=f_1(z):=z/(1-kz^2)$, with $a_3|_{f=f_1}=k$.

Now suppose that $a\neq0$. Denote $t_0:=\max\big\{0,\,\log|2(1+k)/a|\big\}$. Then according to~\eqref{EQ_c1}, $c_1^*(t)=-e^ta/(1+k)$ whenever ${0\le t<t_0}$, and $c_1^*(t)=-2\sgn a$ for all $t\ge t_0$. Substituting $c_1(t):=c_1^*(t)$ into~\eqref{EQ_a}, we get
$$
 a=2k\left(\frac{a t_0}{1+k}+2e^{-t_0}\sgn a\right)=\frac{2k(1+t_0)}{1+k}a.
$$
It follows that $t_0=(1-k)/(2k)$ and
\begin{equation}\label{EQ_a2}
a_2(+\infty)=a/2=\pm\alpha(k),\quad \text{where~$~\alpha(k):=(1+k)e^{-t_0}$.}
\end{equation}
Using~\eqref{EQ_phase1} and~\eqref{EQ_phase2}, we obtain
$$
a_3(+\infty)~=~a_2(+\infty)^2-\displaystyle\int\limits_0^{+\infty}\!\! e^{-2t}p^*_2(t)\,\di t,
$$
where $p_2^*(t)$ is the value of $p_2$ that corresponds to $\big(c_1,c_2\big)=\big(c_1^*(t),c_2^*(t)\big)$. Elementary calculations yield $p_2^*(t)=2k\big(e^{2(t-t_0)}(1+k)-1\big)$ when ${0\le t\le t_0}$, $p_2^*(t)=2k^2$ for all~${t\ge t_0}$, and hence $$a_3(+\infty)=k\big(1+e^{1-1/k}(1+k)\big)>k=a_3|_{f=f_1}.$$

This gives the maximal value of $\Re a_3$ (and hence of~$|a_3|$) in~$\clS_k^B$.
There are two extremal functions for~$\Re a_3$, which we denote by $f_\pm$, corresponding to two possible choices of the sign in~\eqref{EQ_a2}. Since  $z\mapsto -f(-z)$ has the same coefficient~$a_3$ as~$f$, it is clear that $f_{-}(z)=-f_+(-z)$, and the set of all extremal functions for~$|a_3|$ coincides with the rotations of~$f_+$. Therefore, we may assume the sign ``$+$'' in~\eqref{EQ_a2}. Then the same method as in case~$a=0$ allows us to write down the corresponding Heglotz function explicitly,
$$
p(z,t)=\frac{1-k z^2+(1-k)e^{t-t_0}z}{1+kz^2+(1+k)e^{t-t_0}z}~\text{~for $t\in[0,t_0]$}\quad
\text{and}\quad p(z,t)=\frac{1 - kz}{1 + kz}~\text{~for~$t\ge t_0$}.
$$

Unfortunately, it does not seem possible to get an explicit formula for the extremal function~$f_+$ and the Loewner chain generated by the above Herglotz function. However, one can find the Beltrami coefficient of the Becker extension provided by this Loewner chain, see e.g. \cite[Proof of Theorem~2]{Istvan},
\begin{equation}\label{EQ_extremal-Beltrami-coefficient}
\mu(z)=\frac{z^2}{|z|^2}\,\frac{p\big(z/|z|,\log |z|\big)-1}{p\big(z/|z|,\log |z|\big)+1}=
\begin{cases}
-k\dfrac{z^4}{|z|^4}\,\dfrac{\rho(k) + \bar z}{\rho(k) + z},&\text{if $|z|\in\big(1,\,\rho(k)\big)$,}\\[2.75ex]
-k\dfrac{z^3}{|z|^3}, & \text{if $|z|>\rho(k)$,}
\end{cases}
\end{equation}
where $\rho(k):=e^{t_0}=\exp\big((1/k-1)/2\big)$.
\end{proof}

\begin{proof}[\fontseries{bx}\selectfont\protect{Proof of Theorem~\ref{TH_not-extremal}}]
Note that $|a_3|\le|a_3-\alpha a_2^2|+\alpha |a_2|^2$ for any $\alpha\in(0,1)$. The inequality $\max_{\clS_k}|a_3|\le\varrho(k)$ follows therefore from  the Fekete\,--\,Szeg\H{o} Theorem, see e.g. \cite[p.\,104]{Duren}, the well-known estimate~$|a_2|\le 2$ for the class~$\clS$, and  Lehto's Majorant Principle~\cite{Lehto_Majorant}.

To show that the maximum of $|a_3|$ in $\clS_k$ is strictly greater than in~$\clS_k^B$, fix $k\in(0,1)$ and note that for any non-constant holomorphic functional $\Phi:\clS\to\Complex$, according to Lehto's Majorant Principle, the function $q\mapsto\max_{\clS_q}|\Phi|$ is \emph{strictly} increasing. It follows that the extremal functions in the problem $|\Phi|\to\max_{\clS_k}$ do not belong to~$\clS_q$ whenever~${q<k}$. Therefore, to complete the prove, it would be sufficient to show that the Becker q.c.-extension of the function~$f_+$ from Theorem~\ref{TH_a_3} whose Beltrami coefficient is given by~\eqref{EQ_extremal-Beltrami-coefficient} is not extremal, i.e. that $f_+$ admits a $q$-q.c. extension to~$\Complex$ with some $q\in(0,k)$.

Suppose on the contrary that the above mentioned Becker extension of~$f_+$ is extremal. Then it would satisfy the Hamilton\,--\,Krushkal condition~\cite[Theorem~1]{Hamilton}, see also~\cite{Krushkal-extremal,HarringtonOrtel}, which can be formulated as $~\sup_\varphi\big|\Lambda(\varphi)\big|=1$, where
$$
\Lambda(\varphi):={\frac{1}{k}\iint\limits_\Delta\varphi(z)\,\mu(z)\,\di x\di y},\quad \Delta:=\{z:1<|z|<+\infty\},
$$
$\mu$ is given by~\eqref{EQ_extremal-Beltrami-coefficient}, and the supremum is taken over all holomorphic differentials $\varphi(z)\di z^2$ in $\Delta$ with $\|\varphi\|:=\iint_\Delta|\varphi(z)|\di x\di y\,\le\,1$. Note that $\varphi(z)\di z^2$ does not have to be holomorphic at~$\infty$, because the q.c.-extensions of $f_+$ that we consider are required to fix~$\infty$.

The results of~\cite[\S3]{HarringtonOrtel} can be extended without any trouble from~$\UD$ to~$\Delta$. In particular, by~\cite[Proposition~3.2]{HarringtonOrtel}, either $\big|\Lambda(\varphi_*)\big|=1$  for some~$\varphi_*$ with $\|\varphi_*\|=1$ or $\big|\Lambda(\varphi_n)\big|\to1$ as $n\to+\infty$ for some sequence $(\varphi_n)$ with $\|\varphi_n\|\le1$ converging locally uniformly in~$\Delta$ to zero. On the one hand, the former possibility does not hold in our case, because~$\mu(z)$ is not of the form $k\overline{\phi(z)}/|\phi(z)|$, where $\phi$ is holomorphic, see~\cite[p.\,161]{HarringtonOrtel}. On the other hand, in terms of the Laurent development $\varphi_n(z)=\sum_{m=3}^{+\infty}c_{n,m}z^{-m}$, we have
$$
\frac{\Lambda(\varphi_n)}{2\pi}=\frac{1+\log\rho(k)}{\rho(k)}c_{n,3}~+~\frac{\rho(k)^2-2\log\rho(k)-1}{2}~\sum_{m=4}^{+\infty}
\frac{(-1)^m c_{n,m}}{\rho(k)^{m-2}}~\text{~}\longrightarrow~0~\text{~}\text{as~}~n\to+\infty,
$$
because for a fixed $r\in\big(1,\rho(k)\big)$ the Cauchy estimates give $|c_{n,m}|\le r^{m}\max_{|z|=r}|\varphi_n(z)|$.
We obtained a contradiction, which shows that $f_+$ has a $q$-q.c. extension to~$\C$ with~${q\in(0,k)}$, and hence the proof is complete.
\end{proof}

\section{Extremal Becker extensions}\label{SS_examples}
Recall that a q.c.-extension~$F:\C\to\C$ of a function~$f\in\clS$ is called \textsl{extremal}, if for any q.c.-extension~$G:\Complex\to\Complex$ of~$f$ we have $\esssup_{|z|>1}|\mu_G(z)|\ge\esssup_{|z|>1}|\mu_F(z)|$, where $\mu_G$ and $\mu_F$ stand for the Beltrami coefficients of $G$ and $F$, respectively. If the equality occurs in the above inequality only for~$G=F$, then $F$ is said to be the \textsl{uniquely extremal} q.c.-extension of~$f$ to~$\Complex$.

There is a simple sufficient condition for a q.c.-extension~$F:\C\to\C$ to be uniquely extremal.  A \textsl{(regular) Teichm\"uller mapping} of a domain~$D$ is a q.c-mapping $F:D\to\ComplexE$ such that $\mu_F(z)=k\overline{\varphi(z)}/{|\varphi(z)|}$ for a.e. $z\in D$, where $k\in(0,1)$ and ${\varphi(z)\,\di z^2}$, ${\varphi\not\equiv0}$, is a holomorphic quadratic differential in~$D$. It is known that~\cite[Theorem~4]{Strebel:1978} if a q.c.-extension of $f\in\clS$ to~$\C$ is Teichm\"uller on~$\Delta:=\C\setminus\overline\UD$ with $\varphi$  satisfying $\|\varphi\|:={\iint_\Delta|\varphi(z)|\,\di x\di y<+\infty}$, then $F$ is uniquely extremal.

\begin{remark}
If $\varphi$ is holomorphic in~$\Delta$ and has a zero of order at least four at~$\infty$, then a q.c.-map of $\Delta$ with the Beltrami coefficient $k\overline\varphi/|\varphi|$ is a Teichm\"uller mapping of the \textit{simply connected} domain~$\ComplexE\setminus\overline\UD$. For this case, certain conditions weaker than $\|\varphi\|<+\infty$ are sufficient for (unique) extremality, see e.g. \cite{Huang, YaoF, YaoP} and references therein.
\end{remark}

Using the above mentioned sufficient condition, we construct a quite large family of functions $f\in\clS_k$ with uniquely extremal extensions obtained via Becker's construction. The idea comes from the following example. Consider the function ${f_\sigma\in\clS}$, ${\sigma\in(0,2)}$, obtained by composing $\UH:=\{z\colon \Re z>0\}\ni \zeta\mapsto\zeta^\sigma$, $1\mapsto 1$, with suitable Moebius transformations. This function admits a unique $|\sigma-1|$\,-\,q.c. extension $F_\sigma:\C\to\C$ and belongs to~$\clSt{|\sigma-1|}$, see \cite[Example~2]{Istvan}.
The Beltrami coefficient of~$F_\sigma$ is $\mu(z)={(\sigma-1)\overline{\varphi(z)}/|\varphi(z)|}$, ${\varphi(z):=1/(z^2-1)^2}$, for all $z\in\Delta$, which can be written as $\mu(\rho\zeta)=\zeta^2\,\psi_\rho(\zeta)$ for all $\rho>1$ and $\zeta\in\partial\UD$, where $\psi_\rho(\zeta):=(\sigma-1)(\zeta^2-1/\rho^2)/(1-\zeta^2/\rho^2)$. The latter means that $F$ is Becker's q.c.-extension~\eqref{EQ_BeckerExt} with the Herglotz function $p(z,t):=\big(1-\psi_{e^t}(z)\big)/\big(1+\psi_{e^t}(z)\big)$. Note that, up to the factor~$(\sigma-1)$, $\psi_\rho$ is a Blaschke product. It turns out that any finite Blaschke product gives rise to a similar example.

\begin{proposition}\label{PR_examples}
 Let $k\in(0,1)$, $n\in\Natural$, $a_1,\ldots, a_n\in\UD$, $\alpha\in\Real$. Then the Herglotz function
 \begin{equation}\label{EQ_examples_p}
  p(z,t):=\frac{1+k\psi_t(z)}{1-k\psi_t(z)},~\text{~where~}~\psi_t(z):=e^{i\alpha}\, \prod_{j=1}^n\frac{z-e^{-t}a_j}{1-e^{-t}\,\overline a_j z},\quad z\in\UD,~t\ge0,
 \end{equation}
 satisfies Becker's condition~\eqref{EQ_Becker-condition} and formula~\eqref{EQ_BeckerExt} defines a uniquely extremal $k$-q.c. extension of the function~$f\in\clS$ generated by~$p$. In particular, $f\in\clSt{k}\setminus\clS_k^B$ if $a_k\neq0$ for all $k=1,\ldots, n$; otherwise, $f\in\clS_k^B$.
\end{proposition}
\begin{proof}
Condition~\eqref{EQ_Becker-condition} holds trivially because for all~$t\ge0$, $\psi_t$ is a Blaschke product.
We can find the Beltrami coefficient of the $k$-q.c. extension~$F$ given by~\eqref{EQ_BeckerExt}, see e.g.~\cite[\S4]{Istvan},
$$
\mu_F(\rho\zeta)=\frac{p(\zeta,\log\rho)-1}{p(\zeta,\log\rho)+1}\,\zeta^2\,= \,k\frac{\,\overline{\varphi(\rho\zeta)}\,}{|\varphi(\rho\zeta)|},
~\text{~where~} \varphi(z):={e^{-i\alpha}z^{n-2}}\,\prod_{j=1}^n\frac{1}{(z-a_j)^2},~ z\in\Delta,
$$
for all~$\rho>1$ and $\zeta\in\partial\UD$. Hence $F|_{\C\setminus\overline\UD}$ is a Techm\"uller mapping. Moreover, it is easy to see that~$\|\varphi\|<+\infty$. Therefore, $F$ is the uniquely extremal q.c.-extension of~$f$ to~$\C$.

To complete the proof it remains to notice that the normalization ${p(0,t)=1}$ for a.e.~${t\ge0}$ holds only if at least one of the points~$a_k$ coincides with the origin.
\end{proof}

\begin{remark}
Recently, using the generalization of Becker's construction due to Betker~\cite{Betker}, Sugawa~\cite{Sugawa2019} established a sufficient condition for a Beltrami coefficient in~$\UD$ to be \hbox{\textsl{trivial},} i.e. to be the Beltrami coefficient of some q.c.-automorphism of~$\UD$ whose continuous extension to~$\overline\UD$ coincides on~$\partial\UD$ with the identity map.
There is a natural one-to-one correspondence between Beltrami coefficients~$\nu\in L^{\infty}(\UD)$ satisfying Sugawa's condition and Becker's q.c.-extensions. In particular, the $k$-q.c. extension of~$f$ defined in Proposition~\ref{PR_examples}  corresponds to~$\nu(e^{-t}\zeta)= k\zeta^2\psi_t(\zeta)= k\,\phi(e^{-t}\zeta)/|\phi(e^{-t}\zeta)|$ for all $t>0$ and $\zeta\in\partial\UD$, where~ $\phi(z):=e^{i\alpha}z^{n+2}\big/\prod_{j=1}^n(1-\overline a_j z)^2$, $z\in\UD$. This resembles Teichm\"uller mappings except that $\phi(z)$ in the numerator does not carry conjugation.
\end{remark}

\section{Relation between classes $\clS_k$ and $\clS_k^B$}\label{SS_relation}
Although $\clS_k^B$ represents only a part of $\clS_k$, see e.g.~\cite[\S5]{Istvan}, it is plausible to believe that Becker extendible mappings should have yet undiscovered but essential role for the study of conformal mappings admitting quasiconformal extensions.

First of all, functions of the form $f_n(z):=z/(1-ke^{-i\theta}z^n)^{2/n}$, $n\in\Natural$, $\theta\in\Real$, seem to play an important role in  extremal problems for~$\clS_k$, similar to that of the Koebe function $f(z):=z/(1-z)^2$ for the whole class~$\clS$. In fact, $f_1$ and $f_2$ are to known to be extremal in some classical problems, see e.g. \cite{Kuhnau69,Lehto71}.  It is not difficult to see that $f_n\in\clS_k^B$ for all~$n\in\Natural$. Moreover, according to~Proposition~\ref{PR_examples}, there is an infinite family of functions $f\in \clS_k$ for which the uniquely extremal quasiconformal extension to~$\Complex$ is a Becker extension and hence $f\in\clS_k^B\,\big\backslash\bigcup_{0<\nu<k}\,\clS_{\nu}$.

Secondly, there exists $k_*\in(0,1]$ such that for any~$k\in(0,k_*)$ we have $\clS_k\subset\clS_q^B$ with some $q\in(0,1)$ depending only on~$k$. In fact, it is easy to see that $k_*\ge1/6$. Indeed, on the one hand, $|f''(z)/f'(z)|\le 6(1-|z|^2)$ for all $z\in\UD$ and any $f\in\clS$, see e.g. \cite[Ch.\,II,\,\S4, ineq.\,(6)]{Goluzin}, with $6$ replaced by $6k$ if $f\in\clS_k$ thanks to Lehto's Majorant Principle, see e.g. \cite[\S22]{Lawrynowicz}. On the other hand, if a holomorphic function $f:\UD\to\Complex$ satisfies $|f''(z)/f'(z)|\le k(1-|z|^2)$ for all $z\in\UD$, then $f\in\clS_k^B$, see \cite[Satz~4.1]{Becker72}.

We are able to improve slightly the estimate $k_*\ge 1/6$, see Corollary~\ref{CR_AW-Becker0}. In this connection, it is natural to put forward the following problem.
\begin{problem}\label{PRO_strong}
Find $k_*$. In particular, is it true that $k_*=1$, i.e. that for any $k\in(0,1)$ there exists $q\in(0,1)$ such that $\clS_k\subset\clS_q^B$?
\end{problem}

It seems interesting to consider also a bit weaker version of the latter question.
\begin{problem}\label{PRO_weak}
Is it true that for any function $f\in\clS$ admitting a q.c.-extension to~$\Complex$, there exists~$q\in(0,1)$, possibly depending on~$f$, such that $f\in\clS_q^B$?
\end{problem}

Note that it is possible to replace $\clS_k^B$ with $\clSt{k}$ in the above problems as the following proposition shows.
\begin{proposition}
For any $k\in(0,1)$, $\clSt{k}\subset\clS^B_{\kappa(k)}$, where $\kappa(k):=2k/(1+k^2)$.
\end{proposition}
\begin{proof}
If a function $f\in\clSt{k}$ is generated by a Herglotz function $p$ satisfying~\eqref{EQ_Becker-condition}, then the identity
$$
p_0\big(e^{i\Im Q(t)}z,\,\Re Q(t)\big)=L_t(p(z,t)),{~}{\text{~where~}}{~}L_t(z):=\frac{p(z,t)-i\Im p(0,t)}{\Re p(0,t)},
$$
and $Q(t):=\int_0^t\!p(0,s)\di s$, defines a Herglotz function $p_0$ that
obeys the normalization $p_0(0,t)=1$ for a.e.\,$t\ge0$ and, moreover, generates the same function~$f$. The latter can be verified using the change of variables $\tau:=Q(t)$, $\omega(\tau):=e^{i\Im Q(t)}w(t)$ that transforms the Loewner\,--\,Kufarev ODE~\eqref{EQ_LK-ODE0} to $\di \omega/\di\tau=-\omega p_0(\omega,\tau)$.

Note that $L^{\!\UD}_t:=H\circ L_t\circ H^{-1}$, where $H(\zeta):=(\zeta-1)/(\zeta+1)$, is an automorphism of~$\UD$ that sends $z_0(t):=H\big(p_t(0,t)\big)$ to~$0$. Taking into account that by~\eqref{EQ_Becker-condition}, $H\big(p(\UD,t)\big)\subset k\overline\UD$ for a.e.\,$t\ge0$, we see that $H\big(p_0(\UD,t)\big)=L_t^{\!\UD}\big(H\big(p(\UD,t)\big)\big)$ is contained for a.e.\,$t\ge0$ in $\kappa\overline\UD$, where $\kappa:=2k/(1+k^2)$. The conclusion of the proposition follows immediately.
\end{proof}

One natural way to attack the above Problems~\ref{PRO_strong} and~\ref{PRO_weak} would be to propose several constructions of  Loewner chains $(f_t)$ starting from an arbitrary given function~${f_0=f\in\clS_k}$, with images $f_t(\UD)$ being Jordan domains for all~$t\ge0$, and try to find out whether the map $F:\Complex\to\Complex$ defined by~\eqref{EQ_BeckerExt} is quasiconformal for any of these constructions.

Here we examine two quite natural constructions and show that unfortunately, both fail in general.
Fix some locally absolutely continuous function $\omega:[0,+\infty)\to\partial\UD$ and let ${\rho:[0,+\infty)\to[1,+\infty)}$ be a strictly increasing continuous function with ${\rho(0)=1}$ and ${\lim_{t\to+\infty}\rho(t)=+\infty}$.

\vskip1ex
\noindent{\bf Construction~1.} Let $\Phi:\Complex\to\Complex$ be a $k$-q.c. map such that $f_0:=\Phi|_\UD\in\clS$. For~${t\ge0}$, denote by $f^\Phi_t$, the conformal map of $\UD$ onto $\Phi\big(\rho(t)\UD\big)$ normalized by ${f^\Phi_t(0)=0}$, ${\omega(t)(f^\Phi_t)'(0)>0}$. For a suitable choice of the function~$\rho$, the family~$(f^\Phi_t)_{t\ge0}$ is a Loewner chain. Using Courant's Theorem, see e.g. \cite[Theorem~IX.14]{Tsuji}, it is possible to show that formula~\eqref{EQ_BeckerExt} defines a homeomorphism $F$ of~$\C$.
\vskip1ex
\noindent{\bf Construction~2.} Let $f\in\clS_k$. Denote by $g$ the conformal map of $\Complex\setminus\overline\UD$ onto $\Complex\setminus\overline{f(\UD)}$. For~${t\ge0}$, consider the conformal map~$f_t^g$ of~$\UD$, ${f^g_t(0)=0}$, ${\omega(t)\big(f^g_t\big)'(0)>0}$, onto the Jordan domain bounded by~${g\big(\{z:|z|=\rho(t)\}\big)}$. For a suitable choice of the function~$\rho$, the family~$(f^g_t)_{t\ge0}$ is a Loewner chain and the map~$F$ that it generates via~\eqref{EQ_BeckerExt} is a homeomorphism of~$\C$.
\vskip1ex
We will say that the function $\rho$ is \textsl{admissible} in Construction~1 or, respectively, in Construction~2, if the family $(f_t^\Phi)$, or respectively, the family $(f^g_t)$ is a Loewner chain. Note that admissibility of~$\rho$ does not depend on the choice of~$\omega$.

\begin{proposition}
There exists a $(1/\sqrt{2})$-q.c. map $\Phi:\Complex\to\Complex$ with $f:=\Phi|_\UD\in\clS$ such that the homemorphisms~$F$  defined in Constructions~1 and~2 are not quasiconformal for any admissible ${\rho:[0,+\infty)\to[1,+\infty)}$  and any locally absolutely continuous ${\omega:[0,+\infty)\to\partial\UD}$.
\end{proposition}
\begin{proof}
Consider the function
\begin{equation}\label{EQ_f-example}
f(z):=\frac{2z(iz+\sqrt{1-z^2})^i}{1+\sqrt{1-z^2}}=\frac{2ze^{-\arcsin z}}{1+\sqrt{1-z^2}},\quad z\in\UD,
\end{equation}
choosing the unique single-valued branch in~$\UD$ that belongs to~$\clS$. It is not difficult to check that
\begin{equation}\label{EQ_p}
p(z):=\frac{f(z)}{zf'(z)}=\sqrt{\frac{1+z}{1-z}}\quad\text{for all $z\in\UD$}.
\end{equation}
In particular, $\big|\arg p(z)\big|\le\pi/4$. Therefore, by a result of Betker~\cite[p.\,110]{Betker}, see also~\cite[\S5.1]{Ikkei-LM}, $f$ can be extended to a $(1/\sqrt2)$\,-\,q.c. automorphism $\Phi:\Complex\to\Complex$ as follows.

The image $f(\UD)$ is a starlike Jordan domain symmetric w.r.t.~$\Real$ and bounded by two segments of logarithmic spirals. Namely, $\partial f(\UD)={\{2\exp(-\pi/2+|\theta|+i\theta)\colon\theta\in[-\pi,\pi]\}}$. It follows that for any~$z\in\UD\setminus\{0\}$ the intersection of $\partial f(\UD)$ and $\{tf(z):t\ge0\}$ consists of one point $\zeta(z)$, with $r(z):=|\zeta(z)|=2\exp\!\big(|\!\Arg f(z)|-\pi/2\big)$, where $\Arg w$ stands for the value of $\arg w$ that belongs to~$(-\pi, \pi]$.

Betker's q.c.-extension of~$f$, see e.g. \cite[eq.\,(5.6) with~$\lambda:=0$]{Ikkei-LM},  is given by
$$
\Phi(z):=\frac{~r(1/\bar z)^2}{\hphantom{m}\vphantom{\int_0^1}\overline{f(1/\bar z)}\hphantom{m}}=\frac{4e^{-\pi}}{f(1/z)}\left(\frac{f(1/\bar z)}{|f(1/\bar z)|}\right)^{\! -2i\,\eta(z)}\!\!=~\frac{4e^{-\pi}}{f(1/z)}\left(\frac{f(1/z)}{f(1/\bar z)}\right)^{\! i\,\eta(z)}
$$
for all~$z\in\C\setminus\UD$, where $\eta(z):=\sgn\Im z$. Simple calculations give
$$
\frac{\Phi'_z(z)}{\Phi(z)}=\frac{1-i\eta(z)}{z^2}\frac{f'(1/z)}{f(1/z)},\quad %
\frac{\Phi'_{\bar z}(z)}{\Phi(z)}=\frac{i\eta(z)}{\bar{z}^2}\frac{f'(1/\bar z)}{f(1/\bar z)},\qquad |z|>1,~z\not\in\Real.
$$
Using the above formulas we see that for any $r>1$ the boundary of $D_r:=\Phi(r\UD)$ consists of two real-analytic arcs with common end-points at $\Phi(\pm r)$, where they form angle of magnitude $2\arctg(1/2)<\pi/2$. The angle at $\Phi(-r)$ is internal w.r.t. $D_r$. It follows that conformal mappings of $\UD$ onto $D_r$ do not belong to the Hardy space~$H^2(\UD)$. Therefore, by the main result of~\cite{Betker_Hardy}, there is no Loewner chain with image domains~$D_r$ that defines a q.c.-extension via~\eqref{EQ_BeckerExt}. Therefore, the homeomorphism~$F$ in Construction~1 generated by the Loewner chain~$(f^\Phi_t)$ is not quasiconformal, whichever $\rho$ and~$\theta$ we choose.

\stackMath
\newcommand{\fdot}{\vphantom{f_t}\smash{\stackengine{-.075ex}{f_t}{\boldsymbol{\cdot\hskip.175em}}{O}{r}{F}{\useanchorwidth}{S}}}

Let us now consider Construction~2 with the same function ${f\in\clS_{1/\sqrt{2}}}$ as above. One remarkable property of $\Omega:=f(\UD)$ is that $\{1/z\colon z\in\ComplexE\setminus\overline{\Omega}\}=-\tfrac14\Omega.$ It follows that, up to rotation, $g(z)=-4/f(-1/z)$ for all $z\in\Complex\setminus\overline\UD$. Suppose that for a suitable choice of the functions $\rho$ and~$\omega$, the homeomorphism~$F:\Complex\to\Complex$ defined  with the help of the Loewner chain~$(f^g_t)$ is $k$-quasiconformal for some ${k\in(0,1)}$. Then arguing as in~\cite[proof of Theorem~2]{Istvan}, we see that for all ${t\ge0}$ aside from some null-set~$N$, $\fdot:=\partial f_t/\partial t$ and $f_t'$ exist a.e. on~$\partial\UD$, do not vanish, and
\begin{equation}\label{EQ_limit-property}
\frac{\fdot(e^{i\theta})}{e^{i\theta}f_t'(e^{i\theta})}\,\in\,U(k),
\end{equation}
where~$U(k)$ is defined in Theorem~\ref{TH_Becker}. Moreover, by construction, $\partial f_t(\UD)$ is $C^{\infty}$ when~${t>0}$. Hence, in fact, $f_t'$ extends smoothly to~$\partial\UD$  for all~$t>0$; see e.g. \cite[Chapter~3]{Pommerenke:BB}.

Taking into account that $\big|g^{-1}\big(f_t(e^{i\theta})\big)\big|=\rho(t)$ for all~${t\ge0}$ and all~${\theta\in[0,2\pi]}$, it follows that $\rho'(t)$ exists for any~$t\in(0,+\infty)\setminus N$ and for the normal velocity of~$\partial f_t(\UD)$ we have
$$
|f'_t(e^{i\theta})|\,\Re\!\Big(\frac{\fdot(e^{i\theta})}{e^{i\theta}f'(e^{i\theta})}\Big)~=~%
\frac{\Re\big(\fdot(e^{i\theta})\,\overline{e^{i\theta}f'(e^{i\theta})}\,\big)}{|f'_t(e^{i\theta})|}~=~%
\rho'(t)\,\big|g'\big(g^{-1}(f_t(e^{i\theta}))\big)\big|.
$$
Together with~\eqref{EQ_limit-property} this implies that on the one hand, for any~$t\in(0,+\infty)\setminus N$,
\begin{equation}\label{EQ_ravn}
\frac1K\le\left|\frac{\rho'(t)\,g'\big(g^{-1}(f_t(e^{i\theta}))}{f_t'(e^{i\theta})}\right|\le K:=\frac{1+k}{1-k}\quad \text{for all~$\theta\in[0,2\pi]$}.
\end{equation}

On the other hand,
\begin{equation}\label{EQ_integration}
2\pi\rho(t)~=~\int_{0}^{2\pi}\left|\frac{\di g^{-1}(f_t(e^{i\theta}))}{\di \theta}\right|\,\di\theta%
~=~\int_{0}^{2\pi}\left|\frac{f_t'(e^{i\theta})}{g'\big(g^{-1}(f_t(e^{i\theta}))}\right|\,\di\theta
\end{equation}

Combining~\eqref{EQ_ravn} with~\eqref{EQ_integration}, we see that
$$
\frac{\rho(t)}{K^2}\le \left|\frac{f_t'(e^{i\theta})}{g'\big(g^{-1}(f_t(e^{i\theta}))}\right|\le \rho(t) K^2\quad t>0,~t\not\in N.
$$
Therefore, the conformal weldings $\gamma_t:=\big(g^{-1}\circ f_t|_{\partial\UD}\big)\big/\rho(t)$, $t\in(0,+\infty)$, are $K^2$-Lipschitz continuous. Using Carath\'eodory's Extension Theorem (see e.g. \cite[p.\,18]{Pommerenke:BB}) and Courant's Theorem (see e.g. \cite[Theorem~IX.14]{Tsuji}) we conclude that $\gamma_t\to\gamma_0$ as $t\to0^+$. It follows that $\gamma_0$ has to be also Lipschitz-continuous, but in reality it is not. This contradiction shows that $F$ is not quasiconformal.
\end{proof}

\section{A sufficient condition for Becker extendibility}
Below we prove a sufficient condition for a holomorphic function to be Becker extendible, i.e. to have a q.c.-extension of the form~\eqref{EQ_BeckerExt}. This simple result is probably known to specialists: somewhat similar ideas appeared e.g. in~\cite{Betker_Phi} and~\cite[equation\,(11)]{Ikkei1}. However, it does not seem to be ever stated in the form as presented below. For the notions of a meromorphic function of several complex variables and that of an analytic set we refer the reader to~\cite[\S15, \S8]{Shabat}.
\begin{theorem}
\label{TH_PDE_THM}
Let $f$ be a holomorphic function in $\D$, with $f'(0)-1=f(0)=0$.
Suppose that there exists a meromorphic solution $\Phi : \C \times \D \to \C$ to the PDE initial value problem
\begin{align}
 &\Phi'_w(z, \,w ) = \varphi(z,w)\, \Phi'_z(z, \,w), & (z,w)\in \C \times \D;\label{EQ_PDE1}\\
 &\Phi(z,z)  = f(z), & z\in\UD,\label{EQ_PDE2}
\end{align}
with a coefficient $\varphi$ meromorphic in~$\C \times \D$ and satisfying the following two conditions:
\begin{itemize}
\item[(i)] $\varphi(0,0)=0$;
\item[(ii)] $r\!\left|\varphi(w/r,w)\right| \le k$ for all $w\in\D$ and all $r\in\big(|w|^2,1\big)$.
\end{itemize}
Suppose also that there exists $\varepsilon\in(0,1)$ and $M>0$ such that
\begin{equation}\label{EQ_Phi-at-infinity}
|\Phi(z,w)|\le M |z|\quad\text{whenever~$|w|\le |z|$ and $|z\,w|\le \varepsilon^2$.}
\end{equation}
Then $f$ admits a $k$-q.c. Becker extension given by
\begin{equation}
F(z) := \Phi(z, 1/\bar{z}), \hspace{15pt} |z|>1.
\end{equation}
In particular, $f\in\clS_k^B$.
\end{theorem}

\begin{remark}
Since $\Phi(0,0)=f(0)=0$, it is sufficient to check condition~\eqref{EQ_Phi-at-infinity} only for~$z$ large enough.
\end{remark}

Before proving Theorem~\ref{TH_PDE_THM}, let us consider a few examples.

\begin{example}
Let $f$ be a holomorphic function in~$\UD$ with $f'(0)-1=f(0)=0$. Set $\varphi(z,w):=(z-w)f''(w)/f'(w)$. Then $\Phi(z,w):=f(w)+(z-w)f'(w)$ solves problem~\eqref{EQ_PDE1},\,\eqref{EQ_PDE2} and  satisfies~\eqref{EQ_Phi-at-infinity}.  Condition~(i) in Theorem~\ref{TH_PDE_THM} holds trivially, while~(ii) is equivalent to $(1-|w|^2)|wf''(w)/f'(w)|\le k$, which is a classical sufficient condition for q.c.-extendibility.
\end{example}
\begin{example}
Similarly, setting $\varphi(z,w):=f'(w)-1$ and $\Phi(z,w):=f(w)+z-w$, we recover another well-known sufficient condition for q.c.-extendibility $\big|f'(w)-1\big|\le k$, $w\in\UD$, see \cite[\S3]{Brown}.
\end{example}

The following corollary represents another example.
\begin{corollary}\label{CR_AW-Becker}
Fix $k\in(0,1)$. Let $f(z)=z+a_2 z^2+\ldots$ be holomorphic in~$\UD$. If
\begin{equation}\label{EQ_AW-Becker}
\tfrac{4\sqrt3}{9}(1-|z|^2)|a_2|\,+\,(1-|z|^2)^2\big|a_2^2+\tfrac12S_f(z)\big|~\le~ k\quad\text{for all~$z\in\UD$},
\end{equation}
then $f\in\clS_k^B$, with its Becker extension given by $F(z)=\Phi(z,1/\bar z)$ for all~$z\in\Complex\setminus\overline\UD$, where
\begin{equation}\label{EQ_Phi}
\Phi(z,w):=f(w)+\frac{f'(w)}{\frac{1}{z-w}+a_2-\frac{1}{2}\frac{f''(w)}{f'(w)}}.
\end{equation}
\end{corollary}
\begin{proof}
Let $\varphi(z,w):=2a_2(z-w)+(z-w)^2\big(a_2^2+\tfrac12S_f(w)\big)$, where $S_f$ stands for the Schwarzian derivative of~$f$.
Then $\Phi$ given by~\eqref{EQ_Phi} solves problem~\eqref{EQ_PDE1},\,\eqref{EQ_PDE2}.

Moreover, there exists $K>1$ such that $\big|f'(w)\big|\le K$ and $\left|a_2-\frac12\big({f''(w)}/{f'(w)}\big)\right|\le K|w|$  for all~$w\in\frac12\UD$. Hence,
for any $(z,w)\in\C\times\UD$ with $|w|\le |z|$ and ${|zw|\le\varepsilon^2}:=(4K)^{-1}$,
$$
 \big|\Phi(z,w)\big|\,\le\,\big|f(w)\big|\,+\,\frac{\big|(z-w)f'(w)\big|}{\left|1 - K|w|\cdot|z-w|\vphantom{\int_0^1}\right|}%
 \,\le\, K|w|\,+\,\frac{2\,|z|\,\big|f'(w)\big|}{1 - 2K\varepsilon^2\,}%
 \,\le\, 5 K|z|.
$$
This proves~\eqref{EQ_Phi-at-infinity}. Finally, since $|w|(1-|w|^2)\le 2\sqrt 3/9$ for all~$w\in\UD$, condition~\eqref{EQ_AW-Becker} ensures that $\varphi$ satisfies~(ii), while~(i) holds trivially.

Thus, the desired conclusion takes place due to~Theorem~\ref{TH_PDE_THM}.
\end{proof}

\begin{remark}
A well-known result by Ahlfors and Weill~\cite{AW62}, see also~\cite{Ahlfors74}, asserts that if a holomorphic function $f:\UD\to\Complex$ satisfies $\frac12(1-|z|^2)^2|S_f(z)|\le k$, where $k\in(0,1)$, for all~$z\in\UD$, then $f$ is univalent and extends to a $k$-q.c. automorphism $F:\ComplexE\to\ComplexE$, with $F$ given by an explicit formula. This extension can be obtained with the help of Becker's construction (see, e.g.,~\cite[\S4]{Becker72} and \cite{Becker80}), but it does not have to fix~$\infty$ and hence $F$ is not a Becker extension in general (which was overlooked in \cite[\S5]{Istvan}). Corollary~\ref{CR_AW-Becker} is a sort of modification of the Ahfors\,--\,Weill condition that ensures extendibility to a q.c.-automorphism of~$\Complex$.
In fact, if $a_2=0$ then the q.c.-extension of~$f$ given in Corollary~\ref{CR_AW-Becker} coincides with the extension constructed by Ahlfors and Weill~\cite{Ahlfors74}.
\end{remark}

It is known, see e.g. \cite[Example~4 on p.\,132]{Lawrynowicz}, that given $k\in(0,1)$, for all $f\in\clS_k$, $|a_2|\le 2k$ and $|S_f(z)|\le 6k/(1-|z|^2)^2$ for any $z\in\UD$. Therefore, Corollary~\ref{CR_AW-Becker} implies immediately the following statement.
\begin{corollary}\label{CR_AW-Becker0}
If $0<k<0.188856\ldots$\,, then $\clS_k\subset \clS_q^B$ with $q:=(3+\tfrac{8\sqrt3}{9})k+4k^2$.
\end{corollary}
It remains to prove the main result of this section.
\begin{proof}[\fontseries{bx}\selectfont\protect{Proof of Theorem~\ref{TH_PDE_THM}}]
Let $$f_t(\zeta):=\Phi(\zeta e^t,\zeta e^{-t})\quad \text{and}\quad p(\zeta,t):=\frac{\partial f_t(\zeta)/\partial t}{\zeta f_t'(\zeta)}\quad\text{for all $t\ge0$ and~$\zeta\in\UD$.}$$ Thanks to condition~\eqref{EQ_Phi-at-infinity}, these functions are holomorphic in~$\zeta\in\varepsilon\UD$ and real-analytic in~$t\ge0$, and moreover,
$|f_t(\zeta)|\le M e^t$ for all~$\zeta\in\varepsilon\UD$ and $t\ge0$.

Note also that for any fixed $t\ge0$, the point $(\zeta e^t,\zeta e^{-t})$ can lie in the polar set~$\mathcal P$ of~$\Phi$ only for~$\zeta$ belonging to a discrete subset of~$\UD$. Otherwise, since $\mathcal P$ is an analytic set in~$\C\times\UD$, we would have that $(\zeta e^t,\zeta e^{-t})\in\mathcal P$ for all~$\zeta\in\UD$, which contradicts~\eqref{EQ_Phi-at-infinity}. Therefore, $p(\cdot,t)$ and $f_t$ are well-defined meromorphic functions in~$\UD$ for each $t\ge0$.

As an elementary calculation shows, for all~$\zeta\in\UD$ and $t>0$,
$$
\frac{1-p(\zeta,t)}{1+p(\zeta,t)}=e^{-2t}\,\frac{\Phi'_w(\zeta e^t,\zeta e^{-t})}{\Phi'_z(\zeta e^t,\zeta e^{-t})}
=
r\,\varphi(w/r,w),
$$
where $r:=e^{-2t}$ and $w:=\zeta e^{-t}$. Trivially, $|w|^2<r<1$. Therefore,  condition~(ii) implies that $p$ is a Herglotz function satisfying Becker's condition~\eqref{EQ_Becker-condition}.

Clearly, $f_t(0)=\Phi(0,0)=f(0)=0$, $t\ge0$.  Moreover, taking into account~(i), we have  $f'_t(0)=e^t\Phi'_z(0,0)+e^{-t}\Phi'_w(0,0)=e^t\Phi'_z(0,0)=e^t f'(0)=e^t$ for all~$t\ge0$.

We see that $(f_t)$ satisfies the hypothesis of Pommerenke's Criterion \cite[Theorem~6.1 on p.\,159]{Pommerenke}. Hence $(f_t)$ is a classical radial Loewner chain. Furthermore, by Theorem~\ref{TH_Becker}, $f=f_0$ admits a $k$-q.c. extension $F:\Complex\to\Complex$ given by the formula
$$
F(e^{t+i\theta})=f_t(e^{i\theta})=\Phi(e^{t}e^{i\theta},e^{-t}e^{i\theta})=\Phi(z,1/\bar z)\quad \text{for all $z:=e^{t+i\theta}\in\Complex\setminus\overline\UD$,}
$$
which was to be proved.
\end{proof}

\section*{Acknowledgement} The authors are grateful to Professor Toshiyuki Sugawa for fruitful discussions on the topic of the present paper and, in particular, for drawing their attention to reference~\cite{Sugawa2019}.

\end{document}